\documentclass[11pt,reqno]{amsart}
\usepackage{graphicx,amsmath,amsthm,verbatim,amssymb,lineno,eucal,titletoc,enumerate}
\usepackage{bbm}
\usepackage{hyperref}
\hypersetup{colorlinks}
\newcommand{\arxiv}[1]{{\tt \href{http://arxiv.org/abs/#1}{arXiv:#1}}}

\newcommand{\old}[1]{}

\DeclareRobustCommand{\SkipTocEntry}[5]{}
\newtheorem{thm}{Theorem}

\newtheorem{lem}[thm]{Lemma}

\theoremstyle{remark}

\numberwithin{counter}{section}

\theoremstyle{definition}
\newtheorem{defn}{Definition}

\def\zero{\mathbf{0}}
\def\one{\mathbf{1}}

\def\N{\mathbb{N}}

\def\eps{\epsilon}

\begin{document}

\title[]{Multi-Eulerian tours of directed graphs}

\author{Matthew Farrell and Lionel Levine}

\thanks{This research was supported by NSF grant \href{http://www.nsf.gov/awardsearch/showAward?AWD_ID=1455272}{DMS-1455272} and a Sloan Fellowship.}
\address{Matthew Farrell, Department of Mathematics, Cornell University, Ithaca, NY 14850. msf235@cornell.edu}
\keywords{chip-firing, critical group, Eulerian digraph, Laplacian lattice, oriented spanning tree, period vector, Pham index, sandpile group}
\address{Lionel Levine, Department of Mathematics, Cornell University, Ithaca, NY 14850. \url{http://www.math.cornell.edu/~levine/}}

\date{September 21, 2015}
\keywords{BEST theorem, coEulerian digraph, Eulerian digraph, Laplacian, Markov chain tree theorem, oriented spanning tree, period vector, Pham index}
\subjclass[2010]{
05C05,  %	Trees
05C20,  %	Directed graphs (digraphs), tournaments
05C30,  %	Enumeration in graph theory
05C45,  %	Eulerian and Hamiltonian graphs
05C50} %Graphs and linear algebra (matrices, eigenvalues, etc.)

\begin{abstract}
Not every graph has an Eulerian tour. But every finite, strongly connected graph has a \emph{multi-Eulerian tour}, which we define as a closed path that uses each directed edge at least once, and uses edges $e$ and $f$ the same number of times whenever $\mathrm{tail}(e)=\mathrm{tail}(f)$. 
This definition leads to a simple generalization of the BEST Theorem. We then show that the minimal length of a multi-Eulerian tour is bounded in terms of the Pham index, a measure of `Eulerianness'.
\end{abstract}

\maketitle

In the following $G=(V,E)$ denotes a finite directed graph, with loops and multiple edges permitted. We assume throughout that $G$ is \textbf{strongly connected}: for each $v,w \in V$ there are directed paths from $v$ to $w$ and from $w$ to $v$. An \textbf{Eulerian tour} of $G$ is a closed path that traverses each directed edge exactly once.  
Such a tour exists if and only if the indegree of each vertex equals its outdegree; the graphs with this property are called \textbf{Eulerian}.
The BEST theorem (named for its discoverers: de Bruijn, Ehrenfest, Smith and Tutte) counts the number of such tours.
The purpose of this note is to generalize the notion of Eulerian tour and the BEST theorem to any finite, strongly connected graph $G$. 

\begin{defn} Fix a vector $\pi \in \N^V$ with all entries strictly positive. A \textbf{$\pi$-Eulerian tour} of $G$ is a closed path that uses each directed edge $e$ of $G$ exactly $\pi_{\mathrm{tail}(e)}$ times.
\end{defn}

We will see shortly that every strongly connected $G$ has a $\pi$-Eulerian tour for suitable $\pi$, and count the number of such tours. To do so, recall the BEST theorem counting $\one$-Eulerian tours of an Eulerian directed multigraph $G$. Write $\eps_\pi(G,e)$ for the number of $\pi$-Eulerian tours of $G$ starting with a fixed edge $e$. 

\begin{thm} \emph{(BEST \cite{EdB51,TS})}
%\cite[\textsection 5.6]{Stanley} 
A finite, strongly connected multigraph $G$ has a $\one$-Eulerian tour if and only if the indegree of each vertex equals its outdegree, in which case the number of such tours starting with a fixed edge $e$ is
	\[ \eps_\one(G,e) = 
	\kappa_w \prod_{v \in V} (d_v -1)! \]
where $d_v$ is the outdegree of $v$; vertex $w$ is the tail of edge $e$, and $\kappa_w$ is the number of spanning trees of $G$ oriented toward $w$.
\end{thm}

The \textbf{graph Laplacian} is the $V \times V$ matrix
	\[ \Delta_{uv} = \begin{cases} d_v - d_{vv}, & u=v \\
							-d_{vu} & u \neq v \end{cases} \]
where $d_{vu}$ is the number of edges directed from $v$ to $u$, and $d_v = \sum_u d_{vu}$ is the outdegree of $v$.  Observing that the `indegree$=$outdegree' condition above is equivalent to $\Delta \one = \zero$ where $\one$ is the all ones vector, we arrive at the statement of our main result.

\begin{thm}
\label{t.BETTER}
Let $G=(V,E)$ be a strongly connected directed multigraph with Laplacian $\Delta$, and let $\pi \in \N^V$. Then $G$ has a $\pi$-Eulerian tour if and only if 
	\[ \Delta \pi = \zero. \]  
If $\Delta \pi = \zero$, then the number of $\pi$-Eulerian tours starting with edge $e$ is given by
\[
\epsilon_\pi(G,e)= \kappa_w \prod_{v\in V} \frac {(d_{v}\pi_v-1)!}{(\pi_v!)^{d_{v}-1}(\pi_v-1)!}
\]
where $d_v$ is the outdegree of $v$; vertex $w$ is the tail of edge $e$, and $\kappa_w$ is the number of spanning trees of $G$ oriented toward $w$.
\end{thm}

Note that the ratio on the right side is a multinomial coefficient and hence an integer. 
Next we turn to the proof, which is a straightforward application of the BEST theorem. The same device of constructing an Eulerian multigraph from a strongly connected graph was used in \cite[Theorem~3.18]{AB} to relate the Riemann-Roch property of `row chip-firing' to that of `column chip firing'.

\begin{proof}
Define a multigraph $\tilde{G}$ by replacing each edge $e$ of $G$ from $u$ to $v$ by $\pi_u$ edges $e^1,\ldots,e^{\pi_u}$ from $u$ to $v$. Each vertex $v$ of $\tilde{G}$ has outdegree $d_v \pi_v$ and indegree $\sum_{u \in V} \pi_u d_{uv}$, so $\tilde{G}$ is Eulerian if and only if $\Delta \pi = \zero$.

If $(e_1^{i_1}, \ldots, e_m^{i_m})$ is a $\one$-Eulerian tour of $\tilde{G}$, then $(e_1, \ldots, e_m)$ is a $\pi$-Eulerian tour of $G$. Conversely, for each $\pi$-Eulerian tour of $G$ beginning with a fixed vertex $w$, the occurrences of each edge $e$ in the tour can be labeled with an arbitrary permutation of $\{1,\ldots,\pi_{\mathrm{tail}(e)}\}$ to obtain a $\one$-Eulerian tour of $\tilde G$. Hence
	\[ \eps_\one(\tilde{G},w) = \eps_\pi(G,w) \prod_{v \in V} (\pi_v!)^{d_v}. \]
In particular, $G$ has a $\pi$-Eulerian tour if and only if $\tilde G$ is Eulerian.

To complete the counting in the case when $\tilde G$ is Eulerian, the BEST theorem gives the number of Eulerian tours of $\tilde{G}$ starting with a fixed edge $e^1$ with tail $w$, namely
	\[ \eps_\one(\tilde{G},e^1) = \tilde{\kappa}_w \prod_{v \in V} (d_v \pi_v - 1)!. \]
where
	\begin{equation} \label{e.tildekappa} \tilde \kappa_w = \kappa_w \prod_{v \neq w} \pi_v  \end{equation}
is the number of spanning trees of $\tilde G$ oriented toward $w$,
since each spanning tree of $G$ oriented toward $w$ gives rise to $\prod_{v \neq w} \pi_v$ spanning trees of $\tilde G$.

By cyclically shifting the tour, the number of $\pi$-Eulerian tours of $G$ starting with a given edge $e=(w,v)$ does not depend on $v$. Hence
	\[ \frac{\eps_\pi(G,e)}{\eps_\pi(G,w)} = \frac{1}{d_w}, \qquad \frac{\eps_\one(\tilde{G}, e^1)}{\eps_\one(\tilde{G},w)} = \frac{1}{d_w \pi_w}. \]
We conclude that
	\begin{align*} \eps_\pi(G,e) %&= \eps(\tilde{G},e^1) \pi_w \prod_{v \in V} (\pi_v !)^{-d_v} \\ 
					&= \frac{1}{d_w} (d_w\pi_w) \tilde{\kappa}_w \prod_{v \in V} \frac{(d_v \pi_v -1)!}{(\pi_v !)^{d_v}}
	\end{align*}
which together with \eqref{e.tildekappa} completes the proof.
\end{proof}

The watchful reader must now be wondering, is there a suitable vector $\pi$ with \emph{positive integer} entries in the kernel of the Laplacian? The answer is yes. Following Bj\"{o}rner and Lov\'{a}sz, we say that a vector $\mathbf{p} \in \N^V$ is a \textbf{period vector} for $G$ if $\mathbf{p} \neq \zero$ and $\Delta\mathbf{p}=\zero$. A period vector is \textbf{primitive} if the greatest common divisor of its entries is $1$.

\old{
\begin{defn} \cite{BL92}
Given a graph $G$ with Laplacian $\Delta$, a vector $\mathbf{p}\in\mathbb{N}^{n}$
is called a \textbf{period vector} for $G$ if $\mathbf{p} \neq \zero$ and $\Delta\mathbf{p}=\zero$.
A period vector is \textbf{primitive} if the greatest common divisor of its entries is $1$.
\end{defn}
 
The following lemma sums up some useful properties of period vectors.
}

\begin{lem} \cite[Prop.\ 4.1]{BL92} 
\label{lem:period} 
A strongly connected multigraph $G$ has a unique primitive period vector $\pi_{G}$. All entries of $\pi_G$ are strictly positive, and all period vectors of $G$ are of the form $n\pi_{G}$ for $n=1,2,\ldots$. 
Moreover, if $G$ is Eulerian, then $\pi_{G}=\mathbf{1}$.
\end{lem}

Let $\kappa_v$ denote the number of spanning trees of $G$ oriented toward $v$.
Broder \cite{Broder} and Aldous \cite{Aldous} observed that $\kappa$ is a period vector!  This result is sometimes called the `Markov chain tree theorem'. 
%because it implies that by recording the most recent exit of a random walk from each vertex, one obtains a Markov chain on oriented spanning trees, whose stationary distribution is proportional to uniform after conditioning on the root of the tree (=location of the walker).

\begin{lem}[\cite{Aldous,Broder}]
\label{lem:kerper} $\Delta \kappa = \zero$. 
\end{lem}

Lemmas \ref{lem:period} and \ref{lem:kerper} imply that the vector $\pi=\frac{1}{M_G} \kappa$ is the unique primitive period vector of $G$, where
	\[ M_G = \gcd \{ \kappa_v \,:\, v \in V\} \]
is the greatest common divisor of the oriented spanning tree counts.  
Our next result expresses the minimal length of a multi-Eulerian tour in terms of $M_G$ and the number
	\[ U_G = \sum_{v \in V} \kappa_v d_v \]
of \textbf{unicycles} in $G$ (that is, pairs $(t,e)$ where $t$ is an oriented spanning tree and $e$ is an outgoing edge from the root of $t$). 

%The quantity $\kappa(v)$ can be computed by finding the determinant of the matrix gotten by removing the row and column of $\Delta$ corresponding with $v$, as the Matrix tree theorem [cite] asserts that the two quantities are equal.

\begin{thm}
\label{t.minimal}
The minimal length of a multi-Eulerian tour in $G$ is $U_G/M_G$.
\end{thm}

\begin{proof} The length of a $\pi$-Eulerian tour is $\sum_{v \in V} \pi_v d_v$. By Theorem~\ref{t.BETTER} along with Lemmas~\ref{lem:period} and~\ref{lem:kerper}, there exists a $\pi$-Eulerian tour if and only if $\pi$ is a positive integer multiple of the primitive period vector $\frac{1}{M_G} \kappa$. The result follows.
\end{proof}

A special class of multi-Eulerian tours are the simple rotor walks \cite{PDDK,WLB,HLMPPW,HP,Trung}. In a \textbf{simple rotor walk}, the successive exits from each vertex repeatedly cycle through a given cyclic permutation of the outgoing edges from that vertex.  If $G$ is Eulerian then a simple rotor walk on $G$ eventually settles into an Eulerian tour which it traces repeatedly.  More generally, if $G$ is strongly connected then a simple rotor walk eventually settles into a $\pi$-Eulerian tour where $\pi$ is the primitive period vector of $G$.

Trung Van Pham introduced the quantity $M_G$ in \cite{Trung} in order to count orbits of the rotor-router operation. In \cite{coEulerian} we have called $M_G$ the \textbf{Pham index} of $G$ and studied the graphs with $M_G=1$, which we called \textbf{coEulerian graphs}. The significance of $M_G$ is not readily apparent from its definition, but we argue in \cite{coEulerian} that $M_G$ measures `Eulerianness'. Theorem~\ref{t.minimal} makes this explicit, in that the minimal length of a multi-Eulerian tour depends inversely on $M_G$.


\begin{thebibliography}{9}

\bibitem{Aldous} David Aldous, The random walk construction of uniform spanning trees and uniform labelled trees. \emph{SIAM J. Disc.\ Math.} \textbf{3} 450--465, 1990.

\bibitem{AB} Arash Asadi and Spencer Backman, Chip-firing and Riemann-Roch theory for directed graphs, \emph{Electronic Notes Discrete Math.} \textbf{38}:63--68, 2011. \arxiv{1012.0287}

\bibitem{Broder} Andrei Broder, Generating random spanning trees. \emph{Foundations of Computer Science}, 30th Annual Symposium on, pp. 442-447. IEEE, 1989.

\bibitem{EdB51} T. van Aardenne-Ehrenfest and N. G. de Bruijn, Circuits and trees in oriented linear graphs, \emph{Simon Stevin} \textbf{28}, 203--217, 1951.

\bibitem{BL92} Anders Bj\"orner and L\'aszl\'o Lov\'asz, Chip-firing games
on directed graphs, \emph{J. Algebraic Combin.} Vol 1. 305-328, 1992.

\bibitem{coEulerian} Matthew Farrell and Lionel Levine, CoEulerian graphs, \emph{Proc.\ Amer.\ Math.\ Soc.}, to appear, 2015. \arxiv{1502.04690}

\bibitem{HLMPPW} Alexander E. Holroyd, Lionel Levine, Karola M\'{e}sz\'{a}ros, Yuval Peres, James Propp and David B. Wilson, Chip-firing and rotor-routing
on directed graphs, \emph{In and Out of Equilibrium 2,} \emph{Progress
in Probability}, Vol. 60, 331--364, 2008. \arxiv{0801.3306}.

\bibitem{HP} Alexander E. Holroyd and James G. Propp, Rotor walks and Markov chains, in \emph{Algorithmic Probability and Combinatorics}, American Mathematical Society, 2010.
\arxiv{0904.4507}

\bibitem{PDDK} V. B. Priezzhev, Deepak Dhar, Abhishek Dhar and Supriya Krishnamurthy, Eulerian walkers as a model of self-organised criticality, \emph{Phys.\ Rev.\ Lett.} \textbf{77}:5079--5082, 1996. \arxiv{cond-mat/9611019}

\bibitem{Trung} Trung Van Pham, Orbits of rotor-router operation and stationary distribution of random walks on directed graphs, \emph{Adv.\ Applied Math.} \textbf{70}:45--53, 2015. \arxiv{1403.5875}

\bibitem{TS} W.T. Tutte and C.A.B. Smith, On unicursal paths in a network of degree 4. \emph{Amer.\ Math.\ Monthly}: 233--237, 1941.

\bibitem{WLB} Israel A. Wagner, Michael Lindenbaum and Alfred M. Bruckstein, Smell as a computational resource --- a lesson we can learn from the ant, \emph{4th Israeli Symposium on Theory of Computing and Systems}, pages 219--230, 1996.


\end{thebibliography}
\end{document}